\documentclass[12pt]{amsart}
\usepackage{amsfonts}
\usepackage{amsthm}
\usepackage{amsmath}
\usepackage{amscd, hyperref}
\usepackage[latin2]{inputenc}
\usepackage{t1enc}
\usepackage[mathscr]{eucal}
\usepackage{indentfirst}
\usepackage{graphicx}
\usepackage{graphics}
\usepackage{pict2e}
\usepackage{epic}
\numberwithin{equation}{section}
\usepackage[margin=2.9cm]{geometry}
\usepackage{epstopdf} 
\usepackage{enumitem}
\usepackage{tikz}
\usepackage{tikz-cd}
\usetikzlibrary{positioning}
\usepackage{subfig}
\usepackage{comment}
\usepackage{float}
\usepackage{xcolor}

\DeclareMathSymbol{\shortminus}{\mathbin}{AMSa}{"39}

\theoremstyle{plain}
\newtheorem{Th}{Theorem}[section]
\newtheorem{Lemma}[Th]{Lemma}

\newtheorem{Prop}[Th]{Proposition}

 \theoremstyle{definition}
\newtheorem{Def}[Th]{Definition}

\newtheorem{Rem}[Th]{Remark}
\newtheorem{?}[Th]{Problem}

\newcommand{\Grid}{\mathbb{G}}
\newcommand{\Omark}{\mathbb{O}}
\newcommand{\Xmark}{\mathbb{X}}

\newcommand{\cgh}{cGH^{-}}

\begin{document}

\title{A Combinatorial Description of the knot concordance invariant epsilon}

\author[Dey]{Subhankar Dey}
\author[Do\u ga]{Hakan Do\u ga}
\address[]{Department of Mathematics, University at Buffalo}
\email{subhanka@buffalo.edu}
\address[]{Department of Mathematics, University at Buffalo}
\email{hakandog@buffalo.edu}

\begin{abstract} In this paper, we give a combinatorial description of the concordance invariant $\varepsilon$ defined by Hom, prove some properties of this invariant using grid homology techniques. We also compute $\varepsilon$ of $(p,q)$ torus knots and prove that $\varepsilon(\Grid_+)=1$ if $\Grid_+$ is a grid diagram for a positive braid. Furthermore, we show how $\varepsilon$ behaves under $(p,q)$-cabling of negative torus knots.
\end{abstract}

\maketitle

\section{Introduction} 

While a significant progress has been made in recent years, it still remains largely open in low-dimensional topology and knot theory to understand the structure of the knot concordance group. Concordance invariants obtained from various versions of knot Floer homology have proved to be very efficient in further analyzing and understanding the structure of the concordance group. Ozsv\' ath and Szab\'o defined the concordance invariant $\tau$ in \cite{ozsvath4ballgenus} and prove that $\tau$ gives a lower bound for the 4-ball genus of a knot $K$. In the same paper, they prove several useful properties of $\tau$ and compute $\tau$ for certain family of knots. Another important concordance invariant $\Upsilon$ was defined in \cite{upsilondefn} by Stipsicz, Ozsv\' ath and Szab\'o. Using $\Upsilon$, they prove that $\mathcal{C}_{TS}$, which is the subgroup of the smooth concordance group generated by topologically slice knots, contains a $\mathbb{Z}^{\infty}$ summand. Finally, we turn our focus to another concordance invariant $\varepsilon$ that is defined using $CFK^{\infty}(K)$ by Hom in \cite{hom2011knot}. In \cite{hom2014bordered}, Hom uses $\varepsilon$ to calculate the $\tau$ invariant of $(p,q)$ torus knots and $(p,q)$ cables of a knot $K$. \vspace{0.1cm}

Our main goal in this paper is to provide a combinatorial description for $\varepsilon$ using grid homology, compute its value for $(p,q)$ torus knots and positive braids, and finally examine its behavior under $(p,q)$ cabling on torus knots.\footnote{Even though it was not written explicitly
anywhere, the result about positive braids also follows from \cite{positivebraid} and \cite{surveyhomconcordance}. We reprove these results only using grid homology.} 

In an upcoming work, we generalize $\varepsilon$ for knots in lens spaces using the grid description defined by Baker-Grigsby-Hedden in \cite{lensspace}. We will also examine the effect of $(p,q)$-cabling on $\varepsilon$ for any non-trivial knot $K \subset S^3$ and prove similar results for $K \subset L(p,q)$.

The combinatorial treatment of knot Floer homology, namely grid homology, proved to be a valuable tool, especially for computational purposes. Initially defined by Manolescu-Ozsv\' ath-Sarkar in \cite{manolescu2009combinatorial}, it was used to give a simpler proof of the Milnor conjecture \cite{ozsvath2015grid}, to study transverse and Legendrian knots in $S^3$ (\cite{legtransversegrid}, \cite{transversesimplicitygrid}), as well as to compute the concordance invariant $\tau$ for several knots and knot families \cite{gridtau}. Following this idea, we will use grid diagrams to give a description for the concordance invariant $\varepsilon$. Let $\mathbb{G}$ denote a grid diagram for a knot $K$, then the first main result can be stated as follows;

\begin{Th} \label{main} $\varepsilon(\mathbb{G})$ defined via grid homology is a concordance invariant, in the sense that it satisfies the following properties;
\begin{enumerate}[label=\alph*)]
    \item If $K$ is slice, then $\varepsilon(\mathbb{G}) = 0$,
    \item $\varepsilon(\shortminus \mathbb{G}) = \shortminus \varepsilon(\mathbb{G})$ where $\shortminus \mathbb{G}$ denotes the diagram for the mirror reverse of $K$,
    \item If $\varepsilon (\mathbb{G}_1)= \varepsilon (\mathbb{G}_2)$ for some knots $K_1$ and $K_2$, then $\varepsilon(\mathbb{G}_1 \# \mathbb{G}_2)=\varepsilon(\mathbb{G}_1)$.
    \item If $\varepsilon(\mathbb{G}_1)=0$, then $\varepsilon(\mathbb{G}_1 \# \mathbb{G}_2)= \varepsilon(\mathbb{G}_2)$.
\end{enumerate}
\end{Th}

After establishing the above properties which immediately imply that this is a concordance invariant, we return to some computations. We have the following result for negative torus knots in $S^3$, 

\begin{Th} \label{torus}
Let $\mathbb{G}_{-p,q}$ denote the standard grid diagram for the negative $(-p,q)$-torus knot with grid index equal to $|p|+|q|$. Then;
\[ \varepsilon(\mathbb{G}_{-p,q})= \begin{cases} 
      -1 & q>1 \\
      0 & |q|=1 \\
      1 & q<-1 
   \end{cases}
\]

\end{Th}

For the theorem above, we will use the standard grid diagram for negative torus knots used in \cite{ozsvath2015grid} in the proof of Milnor conjecture. In \cite{hom2014bordered}, bordered Heegaard Floer homology is used to prove the behaviour of $\varepsilon$ under $(p,q)$-cabling. We will use grid homology to prove a similar result, yet for only $(p,q)$-cables of negative torus knots. 

\begin{Th} \label{cable}
Let $\mathbb{G}$ denote the grid diagram for $(-p',q')$ torus knot and let $\mathbb{K}_{p,q}$ be a grid diagram for a $(p,q)$-cable of $\mathbb{G}$. Then;
\begin{itemize}
    \item If $\varepsilon(\mathbb{G})\neq 0$, then $\varepsilon({\mathbb{K}_{p,q}})= \varepsilon(\mathbb{G})$.
    \item If $\varepsilon({\mathbb{G}}) = 0$, then $\varepsilon({\mathbb{K}_{p,q}})= \varepsilon(\mathbb{G}_{p,q})$.
\end{itemize}
\end{Th}

Furthermore, we extend our results to positive braids and show that any positive braid has $\varepsilon$ equal to 1. This result relies on positive braids being fibered which is proved by Stallings in more generality in \cite{braidfibered}, for all homogeneous braids, along with Hedden's result in \cite{positivebraid} where he proves $\tau(K)=g$ for fibered, strongly quasipositive knots in $S^3$. We specifically appeal to Proposition 2.1 in \cite{positivebraid} and equivalent statements provided there. As a result, we have the following theorem,

\begin{Th} \label{braid}
Let $\mathbb{G}_+$ denote the grid diagram for a positive braid $B$, then $\varepsilon(\mathbb{G}_+) =1$.
\end{Th}

This paper is structured in the following way; in Section 2, we revisit some definitions from grid homology and describe the method that allows us to compute $\varepsilon$ using grid homology. Section 3 contains the proof of Theorem \ref{main} as well as the method to construct the connected sum of knots using grid diagrams. In Section 4, we compute the value of $\varepsilon$ for negative torus knots and show how $\varepsilon$ behaves under $(p,q)$-cabling. In Section 5, we describe how we can view grid diagrams to represent braids and show that $\varepsilon=1$ for positive braids.

\vspace{0.5cm}

\section{Preliminaries on $\varepsilon(K)$ and Grid Homology:}

Grid homology was defined by Manolescu-Ozsv\' ath-Sarkar in \cite{manolescu2009combinatorial} for knots in $S^3$, using the nice diagrams introduced by Sarkar-Wang in \cite{sarkarwangnicediagrams}. As we will discuss further, the main advantage of grid homology is the simplification of the domains counted by the differential and the combinatorial nature of its setup, despite increasing the complexity of the knot Floer complex. Many different properties of the Grid Homology setting are discussed and extended in \cite{ozsvath2015grid} and we will be following the notation conventions described there. Before we define the combinatorial recipe for $\varepsilon(K)$, we will briefly recall the grid homology setup. 

To construct a grid diagram for a given a knot $K$ in $S^3$, we start by taking a piecewise linear approximation of $K$, then we apply some local modifications so that the vertical strands go over horizontal strands. The resulting diagram can be put on an $n \times n$ grid that is decorated with $\Xmark$ and $\Omark$-markings. Since we can apply these steps for any knot $K$ in $S^3$, any knot can be represented by a grid diagram. In Figure \ref{knottogrid}, an example of this procedure is shown for the Figure-8 knot.

\begin{figure}[h!]
\centering
    \includegraphics[width=12cm]{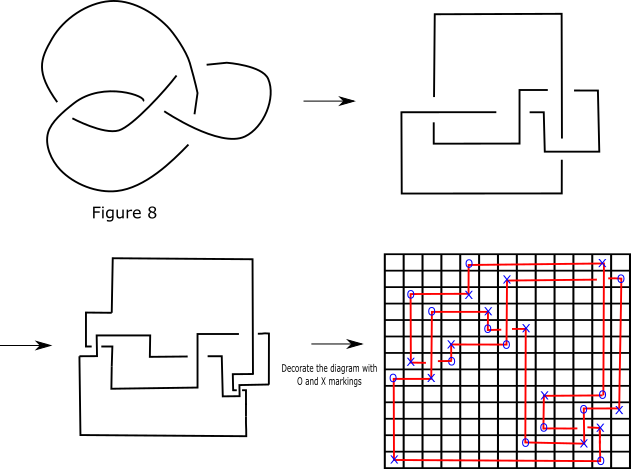}
\caption[PL approximation of a knot and how to obtain a grid diagram]{How to obtain a grid diagram for the Figure-8 knot. Notice that this diagram is $13\times 13$ and we can make it smaller by local isotopies, commutations and destabilizations}
\label{knottogrid}
\end{figure} 

\begin{Def} A grid diagram $\mathbb{G}$ for a knot $K$ in $S^3$ is given by the following data;
\begin{itemize}
    \item $n \times n$ grid decorated with a set of n X-markings and a set of n O-markings that we denote by $\mathbb{X}$ and $\mathbb{O}$, respectively.
    \item Each row and column contains exactly one O and one X marking.
    \item No double markings inside the small squares.
\end{itemize}

\end{Def}

To specify $K$, we draw vertical strands between the markings and orient them from X to O, and then draw the horizontal ones orienting the strands from O to X. We also follow the convention that vertical strands go over horizontal ones. This gives us the grid diagram $\mathbb{G}$ of $K$ with a specified orientation. We call the grid size $n$ the grid index. Furthermore, we identify the top and bottom edges, right-most and left-most edges of the grid diagram, and consider the \textit{toroidal} grid diagram along with the standard counterclockwise orientation on the torus.

\vspace{0.1cm}

A grid state $\mathbf{x}$ is a one-to-one correspondence between horizontal and vertical circles on the grid torus, equivalently an n-tuple $\mathbf{x}=\{x_1, x_2, \dotsc , x_n\}$ where each $x_i$ comes from the intersection points of horizontal and vertical circles. The set of grid states is denoted by $S(\mathbb{G})$. Now given $\mathbf{x}$, $\mathbf{y} \in S(\mathbb{G})$ such that they differ at exactly $n-2$ coordinates, we say $r$ is an empty rectangle from $\mathbf{x}$ to $\mathbf{y}$ and denote the set of empty rectangles by $r \in$ Rect$^0(\mathbf{x}, \mathbf{y})$ if; 
\begin{itemize}
    \item $Int(r) \cap \mathbf{x} = Int(r) \cap \mathbf{y} = \emptyset $ 
    \item $\partial (\partial_h)= \mathbf{y} - \mathbf{x}$ and  $\partial (\partial_v)= \mathbf{x} - \mathbf{y}$ where $\partial_h$ and $\partial_v$ denote horizontal and vertical boundaries respectively. 
\end{itemize}

\vspace{0.1cm}

For a given grid state $\mathbf{x} \in S(\Grid)$, we can define Maslov and Alexander gradings. Both of these gradings can be calculated by a function on the planar realization of the grid diagram. The candidate function is constructed in \cite{ozsvath2015grid} and calculates the Maslov grading as follows for $\mathbf{x} \in S(\mathbb{G})$;

$$ M_{\mathbb{O}} (\mathbf{x})= \mathcal{J}(\mathbf{x},\mathbf{x}) - 2\mathcal{J}(\mathbf{x}, \mathbb{O}) + \mathcal{J}(\mathbb{O}, \mathbb{O}) + 1$$

\vspace{0.2cm}

where $\mathcal{J}(A,B)$ is a symmetrization of the function $\mathcal{I}(A,B)$ which counts the number of elements of $B$ that lie in the North-East of the elements of $A$. Using this, we can also define the following Alexander grading function;

$$A(\mathbf{x})= \frac{1}{2}(M_{\mathbb{O}} (\mathbf{x}) -M_{\mathbb{X}} (\mathbf{x})) - (\frac{n-1}{2}) $$

\vspace{0.2cm}
where $n$ is the grid index. 
\vspace{0.2cm}

For our construction, we will consider a few different versions of the grid homology. Given an $n \times n$ grid diagram $\Grid$, enumerate the $\Omark$-markings for $i=1, \dots, n$ and consider the multipolynomial ring $\mathbb{F}[V_1, \dots, V_n]$ where each $V_i$ corresponds to an $O_i$. The \textit{unblocked grid chain complex} $GC^-(\Grid)$ is freely generated by the set of grid states over $\mathbb{F}[V_1, \dots, V_n]$ and equipped with the following differential,

\begin{equation} \partial^-_{\Xmark} \, \mathbf{x} = \sum\limits_{ \mathbf{y} \in S(\Grid)} 
\sum_{\substack{ \{ r \in Rect^o(\mathbf{x}, \mathbf{y}), \, \,  r \cap \Xmark = \emptyset \}  }} 
 V_1^{O_{1}(r)} \, V_2^{O_{2}(r)} \dots V_n^{O_{n}(r)} \cdot \mathbf{y} \label{differentialgrid}\end{equation}
 
 Multiplication by any $V_i$ viewed as a map from $GC^-(\Grid)$ to itself is of homogeneous degree $(-2,-1)$ which can be written as,
 
 \[ A(V_1^{n_1} \dots V_k^{n_k} \cdot \mathbf{x})= A(\mathbf{x}) - n_1 - \dots - n_k
\]
\[ M(V_1^{n_1} \dots V_k^{n_k} \cdot \mathbf{x})= M(\mathbf{x}) - 2n_1 - \dots - 2n_k
\]

The unblocked grid homology denoted by $GH^-(\Grid)$ is $H_*(GC^-(\Grid), \partial^-_{\Xmark})$. See the example in Figure \ref{gridchain}. In Section 3, we discuss how to construct a grid diagram for the connected sum of two knots from their disjoint union. As a result, we briefly describe a version of grid homology for links called \textit{collapsed grid homology} (See Section 8.2 in \cite{ozsvath2015grid} for further details). Given a grid diagram $\Grid$ for a link $L$ with $l$ components, choose one $V_{j_i}$ for $i=1, \dots , l$ corresponding to a component of $L$. The \textit{collapsed grid complex} is $cGC^-(L)= \frac{GC^-(L)}{V_{j_1} = \dots = V_{j_l}}$ equipped with the same differential in (\ref{differentialgrid}). Homology of this complex, denoted by $cGH^-(L)$, is viewed as a bigraded $\mathbb{F}[U]$-module which is a link invariant. Notice that we also normalize the Alexander grading, 

\begin{equation}
A(\mathbf{x})= \frac{1}{2}(M_{\mathbb{O}} (\mathbf{x}) -M_{\mathbb{X}} (\mathbf{x})) - (\frac{n-l}{2})
\label{alexandergridlink}
\end{equation}

where $n$ is the grid index and $l$ is the number of link components.
 
One other version of grid construction that we will utilize in this work corresponds to the knot Floer complex filtered by the Alexander filtration. This complex is defined in \cite{ozsvath2015grid} and modifies the differential so that we count more rectangles. 

Given a grid diagram $\Grid$ with grid index $n$, \textit{the filtered grid complex} $\mathcal{GC}^-(\Grid)$ is generated by the set of grid states over the polynomial ring $\mathbb{F}[V_1, \dots , V_n]$ and equipped with,

\[ \partial^- \, \mathbf{x} = \sum\limits_{ \mathbf{y} \in S(\Grid)} \, 
\sum_{  r \in Rect^o(\mathbf{x}, \mathbf{y})  } 
 V_1^{O_{1}(r)} \, V_2^{O_{2}(r)} \dots V_n^{O_{n}(r)} \cdot \mathbf{y} \]

Note that this particular differential counts the rectangles that contain $\Xmark$-markings as well. It is proved in \cite{ozsvath2015grid} that this map satisfies $(\partial^-)^2=0$ and respects the Alexander filtration. As a result, $(\mathcal{GC}^-(\Grid), \partial^-)$ is a $\mathbb{Z}$-filtered, $\mathbb{Z}$-graded chain complex over $\mathbb{F}[V_1, \dots, V_n]$.

A more detailed discussion of the properties of this complex such as the invariance of filtered quasi-isomorphism type and the invariants that can be obtained from this complex can be found in \cite{ozsvath2015grid}.

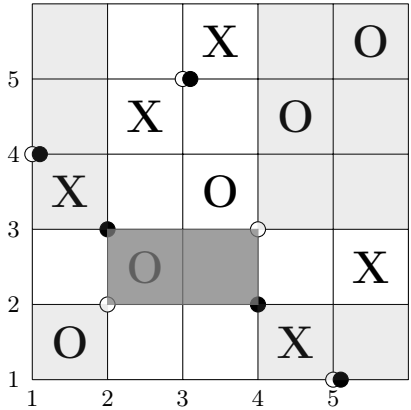
\begin{figure}[h]
\centering
\begin{tikzpicture}

\draw[step=1.00cm,color=black] (0,0) grid (5,5);
\node at (0.5 ,+0.5) { \Large{\textbf{O}}};
\node at (1.5 , 1.5) { \Large{\textbf{O}}};
\node at (2.5, 2.5) { \Large{\textbf{O}}};
\node at (3.5,3.5) { \Large{\textbf{O}}};
\node at (4.5,4.5) { \Large{\textbf{O}}};
\node at (0.5,2.5) { \Large{\textbf{X}}};
\node at (1.5,3.5) { \Large{\textbf{X}}};
\node at (2.5,4.5) { \Large{\textbf{X}}};
\node at (3.5,0.5) {\Large{\textbf{X}}};
\node at (4.5, 1.5) { \Large{\textbf{X}}};
\node at (0, -0.25) { \scriptsize{1}};
\node at (1, -0.25) { \scriptsize{2}};
\node at (2, -0.25) { \scriptsize{3}};
\node at (3, -0.25) { \scriptsize{4}};
\node at (4, -0.25) { \scriptsize{5}};
\node at (-0.25, 0) { \scriptsize{1}};
\node at (-0.25, 1) { \scriptsize{2}};
\node at (-0.25, 2) { \scriptsize{3}};
\node at (-0.25, 3) { \scriptsize{4}};
\node at (-0.25, 4) { \scriptsize{5}};
\draw [fill=white] (1,1) circle [radius=0.1];
\draw [fill=white] (2,4) circle [radius=0.1];
\draw [fill=white] (0,3) circle [radius=0.1];
\draw [fill=white] (3,2) circle [radius=0.1];
\draw [fill=white] (4,0) circle [radius=0.1];
\draw [fill=black] (1,2) circle [radius=0.1];
\draw [fill=black] (3,1) circle [radius=0.1];
\draw [fill=black] (0.1,3) circle [radius=0.1];
\draw [fill=black] (2.1,4) circle [radius=0.1];
\draw [fill=black] (4.1,0) circle [radius=0.1];
\draw [fill=gray,gray, opacity=0.75] (1,1) rectangle (3,2);
\draw [fill=gray,gray, opacity=0.15] (0,0) rectangle (1,1);
\draw [fill=gray,gray, opacity=0.15] (0,2) rectangle (1,5);
\draw [fill=gray,gray, opacity=0.15] (3,2) rectangle (5,5);
\draw [fill=gray,gray, opacity=0.15] (3,0) rectangle (5,1);
\end{tikzpicture}
\caption{Grid diagram for the left-handed trefoil (LHT), two grid states $\mathbf{x}$ and $\mathbf{y}$ are given by empty and full circles respectively. Darker shaded rectangle and lighter shaded rectangle are two rectangles from $\mathbf{x}$ to $\mathbf{y}$. Notice that darker shaded rectangle is counted by the differential whereas lighter shaded one is not since it is not an empty rectangle }
\label{gridchain}
\end{figure}

Now we can also define the concordance invariant $\tau$ which will be helpful for our discussions about $\varepsilon$. Since our aim is to provide a combinatorial description, we follow the definition from grid homology which is equivalent to the definition in the holomorphic theory. For a given knot $K$, let $\mathbb{G}$ be a grid diagram of it. We have the following definition; \newline 

$ \tau(K)=$ - max\{ $n \mid $ \textit{there exists a homogeneous, non-torsion element in} $GH^-(G)$ \textit{with Alexander grading equal to} $n$\}   \newline

We briefly recall the definition of $\varepsilon$ as in \cite{hom2011knot}. Given a knot $K$, consider $CFK^{\infty}(K)$ viewed as a $\mathbb{Z} \oplus \mathbb{Z}$-filtered complex on $(i,j)$-plane, equipped with the differential that counts Whitney disks that crosses both basepoints $z$ and $w$ (See \cite{ozsvath2004holomorphic} for further details). If an element $\mathbf{y}$ appears in the boundary of another element $\mathbf{x}$, we usually draw an arrow from $\mathbf{x}$ to $\mathbf{y}$ on the $(i,j)$-plane. 

Now let $\partial^{vert}$ denote the \textit{vertical differential}, that is the part of the differential of $CFK^{\infty}$ that preserves the $U$-filtration. Similarly, let $\partial^{horz}$ be the \textit{horizontal differential}, the part of the differential that preserves the Alexander filtration. Consider the subcomplex $C\{i=0\}=C^{vert}$ equipped with $\partial^{vert}$ and let $[\mathbf{x}]$ be a homology class that generates $H_*(C^{vert},\partial^{vert})$. We can view this class in the horizontal complex $C\{j=\tau\}=C^{horz}$ equipped with $\partial^{horz}$, and this leads to the definition of $\varepsilon$ as follows,

\begin{itemize}
    \item $\varepsilon(K)=0$ if and only if $[\mathbf{x}]$ is in the kernel, but not in the image of the horizontal differential.
    \item $\varepsilon(K)=1$ if and only if $[\mathbf{x}]$ is in the image of the horizontal differential.
    \item $\varepsilon(K)=-1$ if and only if $[\mathbf{x}]$ is not in the kernel of the horizontal differential.
\end{itemize}

Now, we can define our method to compute $\varepsilon$ using grid homology. 

\vspace{0.3cm}

\textbf{Step 1:} Given an $n\times n$ grid diagram $\Grid$ of $K$, enumerate the $\Omark$-markings for $i=1, \dots , n$ and recall that the rank of the free part of $GH^-(K)$ is 1. We start by choosing a generator $[x_{\tau}]$ of $GH^-(K)$. Notice that this element is not the same as the generator of $H_*(C\{i=0\})$, in fact $A(x_{\tau})=-A(x)$ where $[x]$ is the generator of $H_*(C\{i=0\})$.

\textbf{Step 2:} Instead of looking at the horizontal complex and seeing how the element $x_{\tau}$ interacts with it, we use the symmetry properties of grid homology. The reason we choose this method to compute $\varepsilon$ is to take advantage of the combinatorial nature of the grid homology computations and the fact that diagram based calculations are easier to track once we know which element we would like to follow. Firstly, reflect the grid diagram $\Grid$ with respect to a horizontal axis which yields a grid diagram for the mirror of $K$. We have the following proposition which shows that the generator of the free part is mapped to the generator of the free part of $GH^-(mr(K))$,

\begin{Prop}[\cite{ozsvath2015grid}, Proposition 7.4.3] If $K$ is a knot and $mr(K)$ denotes the mirror of $K$, then the homology of the grid complex of $mr(K)$ is isomorphic to the homology of the dual complex of $K$. More explicitly if $GH^-(K) \cong \mathbb{F}[U]_{(-2\tau,-\tau)} \oplus (Torsion)$ then, $GH^-(mr(K)) \cong \mathbb{F}[U]_{(2\tau,\tau)} \oplus (Torsion)$ where $\tau=\tau(K)$.
\end{Prop}

In the proof of the above proposition, the isomorphism is explicitly constructed with a bijection between grid states and empty rectangles between two diagrams before and after reflection, hence we can keep track of the distinguished element. As Hom explains in \cite{hom2011knot}, this dual complex can be seen geometrically on the filtered grid complex $\mathcal{GC}^-(\Grid)$ by reversing the direction of both filtrations and also reversing the direction of the arrows.

\textbf{Step 3:} After obtaining the mirror of $K$, we also change the orientation by switching the $\Xmark$ and $\Omark$-markings. The result is a grid diagram for the mirror reverse of $K$ which we denoted by $-K$ and the grid diagram that represents the mirror reverse is denoted by $-\Grid$. The effect of changing the orientation is given by the following proposition,

\begin{Prop} [\cite{ozsvath2015grid}, Proposition 7.1.1] \label{mirrorhat} If $K$ is a knot, for all $d,s \in \mathbb{Z}$, $\widehat{GH}_d(K,s) \cong \widehat{GH}_{d-2s}(K,-s)$ where $d$ is the Maslov (homological) grading and $s$ is the Alexander grading.
\end{Prop}

The symmetry above is observed by changing the orientation of the knot. After this orientation reversal, the distinguished element is mapped to the desired Alexander grading in the resulting complex. We also note here that $\widehat{GC}(K) = ( GC^-(K)/V_i \, , \partial^-_{\Xmark})$ is called \textit{simply blocked grid complex}. We describe the boundary map $\widehat{\partial}_{\Xmark,O_n}$ in the next step of our construction. The following lemma ensures that the method described here does not depend on the $O$-marking chosen in the next step.

\begin{Lemma} [\cite{ozsvath2015grid}, Lemma 4.6.9]
For any pair of integers $i,j \in \{ 1, \dots, n \}$, multiplication by $V_i$ is chain homotopic to multiplication by $V_j$ viewed as a map from $GH^-(K)$ to itself.
\end{Lemma}

\textbf{Step 4:} We start by looking at the complex $GC^{-}(\mathbb{G}, \partial^{-}_{\Xmark})$ and choosing a particular $O$-marking, call it $O_n$ We divide the boundary map $\partial^{-}_{\Xmark}$ into two parts where the first summand is $\widehat{\partial}_{\mathbb{X}, O_n}$ and the second summand counts empty rectangles for which $O_n(r) \neq 0$. More concretely,

\[ \partial^-_{\Xmark} \, \mathbf{x} = \sum\limits_{ \mathbf{y} \in S(\Grid)} 
\sum_{\substack{ \{ r \in Rect^o(\mathbf{x}, \mathbf{y}), \, \,  r \cap \Xmark = \emptyset \}  }} 
 V_1^{O_{1}(r)} \, V_2^{O_{2}(r)} \dots V_n^{O_{n}(r)} \cdot \mathbf{y}  \]
 
 \[= \sum\limits_{ \mathbf{y} \in S(\Grid)} 
\sum_{\substack{ \{ r \in Rect^o(\mathbf{x}, \mathbf{y}), \, \,  r \cap \Xmark = \emptyset \\ O_n(r) = 0 \}  }} 
 V_1^{O_{1}(r)} \, V_2^{O_{2}(r)} \dots V_{n-1}^{O_{n-1}(r)} \cdot \mathbf{y}   \]
 \[+ \sum\limits_{\mathbf{y} \in S(G)}  \sum_{ \substack{ \{r \in Rect^o(\mathbf{x},\mathbf{y}), \, r \cap \mathbb{X} = \emptyset, \\ O_n(r) \neq 0\}  } } V_1^{O_{1}(r)} \, V_2^{O_{2}(r)} \dots V_n^{O_{n}(r)} \, \cdot \mathbf{y}\]
 \[ = \widehat{\partial}_{\mathbb{X},O_n} (\mathbf{x}) + \partial_{\Xmark}^{horz}(\mathbf{x}) \]

Recall that in the setting of knot Floer homology, the horizontal complex $C\{j=\tau\}$ is equipped with the horizontal differential which preserves the Alexander grading. Any horizontal arrow changes the $i$-filtration which records the $U$-exponent. In comparison, we write the second summand above as the combinatorial version of the horizontal differential which preserves the Alexander filtration since $\Xmark$-markings are avoided and counts rectangles that contain non-zero $O_n$ multiplicities.

\textbf{Step 5:}  Finally, we can define $\varepsilon$ from the grid homology. Let $[x]$ denote the homology class that the distinguished homology class $[x_{\tau}]$ maps under the above isomorphisms. By Proposition \ref{mirrorhat}, $[x]$ is a non-zero element in $\widehat{GH}(-K)$. Now we consider $(\widehat{GH}(-K), \partial^{horz}_{\Xmark})$ and use that to define $\varepsilon$ combinatorially as follows,

\begin{Def} The above steps yield the following definition of $\varepsilon$,

\begin{itemize}
    \item $\varepsilon(\mathbb{G})=0$ if and only if $[\mathbf{x}] \in$ Ker($\partial_{\Xmark}^{horz}$), but $[\mathbf{x}] \not\in$ Im($\partial_{\Xmark}^{horz}$).
    \item $\varepsilon(\mathbb{G})=1$ if and only if $[\mathbf{x}] \not\in$ Ker($\partial_{\Xmark}^{horz}$).
    \item $\varepsilon(\mathbb{G})=-1$ if and only if $[\mathbf{x}] \in$ Im($\partial_{\Xmark}^{horz}$).
\end{itemize}

\end{Def}

 Following the commonly used nomenclature, we call $\mathbb{F}[U]$ the \textit{tower} part of the grid homology. Observe that an equivalent, alternative definition of $\varepsilon$ can be formulated as follows with the help of the structure of $GH^{-}(-K) \cong \mathbb{F}[U] \oplus Tors$. In particular, we will refer to this alternative definition in some of our computations.

\begin{Def}[Alternative Definition] Given a grid diagram $\Grid$ for a knot $K$,

\begin{itemize}
    \item $\varepsilon(\mathbb{G})=0$ if and only if $[\mathbf{x}]$ is a non-zero element in the \textit{tower} of $GH^{-}(-K)$.
    \item $\varepsilon(\mathbb{G})=1$ if and only if $[\mathbf{x}]$ is zero in $GH^{-}(-K)$.
    \item $\varepsilon(\mathbb{G})=-1$ if and only if $[\mathbf{x}]$ is non-zero, torsion element in $GH^{-}(-K)$.  
\end{itemize}

\end{Def}

Notice that this definition agrees with the definition of Hom in \cite{hom2011knot} with an extra minus sign introduced since we are doing the computations from the minus version of the grid homology. To emphasize that we are doing the computations using the grid diagrams, we denote the invariant as $\varepsilon(\Grid)$. Using the grid diagram for the mirror reverse of $K$ allows us to map the distinguished element to the appropriate Alexander grading where the computations become simpler.

\begin{figure}[h!]
\centering
\includegraphics[ width=0.75 \textwidth]{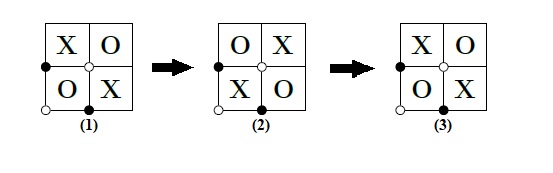}
\caption{From left to right, we start with the standard grid diagram for the unknot $\mathbb{G}_{\mathcal{O}}$, reflect it with respect to a horizontal axis and then change $\mathbb{X}$ and $\mathbb{O}$ markings to obtain $\shortminus \mathbb{G}_{\mathcal{O}}$ } \label{unknot}
\end{figure}

In Figure [\ref{unknot}], we see the example of the unknot $\mathbb{G}_{\mathcal{O}}$. Notice that the grid state given by full circles is a cycle since both rectangles leaving this grid state contain $\mathbb{X}$ markings, and this cycle generates the vertical homology. After applying the above definition to compute $\varepsilon$, we can immediately see that in $\shortminus \mathbb{G}_{\mathcal{O}}$, the same element is non-trivial. Hence, we can conclude that $\varepsilon(\mathbb{G}_{\mathcal{O}})=0$. 

Our goal is to prove that this combinatorial method described above provides a powerful and efficient way to compute $\varepsilon$. To prove the Theorems \ref{torus} and \ref{cable}, we use the standard grid diagram for $(-p,q)$-torus knots. 

To this end, we briefly recall the standard method of encoding any grid diagram using permutations. In the planar realization of a grid diagram, we enumerate the columns starting from left and enumerate the rows starting from the bottom, then we can represent the grid diagram using two permutations that are elements of $S_n$ if the grid index is $n$. These permutations are denoted by $\sigma_{\mathbb{X}}$ and $\sigma_{\mathbb{O}}$ recording the positions of $\mathbb{X}$ and $\mathbb{O}$-markings. We follow the convention that if a marking is in the $i^{th}$ column and $j^{th}$ row, then the permutation takes $i$ to $j$, and for simplicity, we only write $j$ in the permutation. 

In the case of a $(-p,q)$-torus knot, we can form a grid diagram denoted by $\mathbb{G}_{-p,q}$ with grid index $n=|p|+|q|$ and with permutations $\sigma_{\mathbb{O}}=(1, 2, 3, \dots , p+q)$ and $\sigma_{\mathbb{X}}=(p+1, p+2, \dots, p+q , 1, 2, \dots , p)$. See Figure [\ref{negativetorus}] for an example. 

\begin{figure}[h!] 
\centering
\begin{tikzpicture}
\draw[step=1.00cm,color=black] (0,0) grid (5,5);
\node at (0.5 , 0.5) { \Large{\textbf{O}}};
\node at (1.5 , 1.5) { \Large{\textbf{O}}};
\node at (2.5 , 2.5) { \Large{\textbf{O}}};
\node at (3.5 , 3.5) { \Large{\textbf{O}}};
\node at (4.5 , 4.5) { \Large{\textbf{O}}};
\node at (0.5 , 2.5) { \Large{\textbf{X}}};
\node at (1.5 , 3.5) { \Large{\textbf{X}}};
\node at (2.5 , 4.5) { \Large{\textbf{X}}};
\node at (3.5 , 0.5) { \Large{\textbf{X}}};
\node at (4.5 , 1.5) { \Large{\textbf{X}}};
\draw[red, thick] (0.5, 2.5) -- (0.5, 0.5);
\draw[red, thick] (1.5, 1.5) -- (1.5, 3.5);
\draw[red, thick] (2.5, 4.5) -- (2.5, 2.5);
\draw[red, thick] (3.5, 3.5) -- (3.5, 0.5);
\draw[red, thick] (4.5, 1.5) -- (4.5, 4.5);
\draw[red, thick] (2.5, 4.5) -- (4.5, 4.5);
\draw[red, thick] (0.5, 0.5) -- (3.5, 0.5);
\draw[red, thick] (1.5, 1.5) -- (3.25, 1.5);
\draw[red, thick] (3.75, 1.5) -- (4.5, 1.5);
\draw[red, thick] (0.5, 2.5) -- (1.25, 2.5);
\draw[red, thick] (1.75, 2.5) -- (2.5, 2.5);
\draw[red, thick] (1.5, 3.5) -- (2.25, 3.5);
\draw[red, thick] (2.75, 3.5) -- (3.5, 3.5);
\end{tikzpicture}
\caption{Grid diagram $\mathbb{G}_{-2,3}$ for the left-handed trefoil, grid index is 5 and $\sigma_{\mathbb{O}}=(1, 2, 3, 4, 5)$ and $\sigma_{\mathbb{X}}=(3, 4, 5, 1, 2)$ }
\label{negativetorus}
\end{figure}
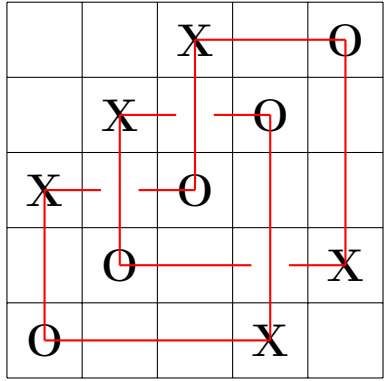

Figure [\ref{comparative}] draws a parallel picture between how we compute $\varepsilon$ using our grid homology method and how $\varepsilon$ can be computed from $HFK^{-}$ in the case of the left-handed trefoil, i.e. $T_{-2,3}$.

\begin{figure}[h!]
    \centering
    \includegraphics[width = 13cm]{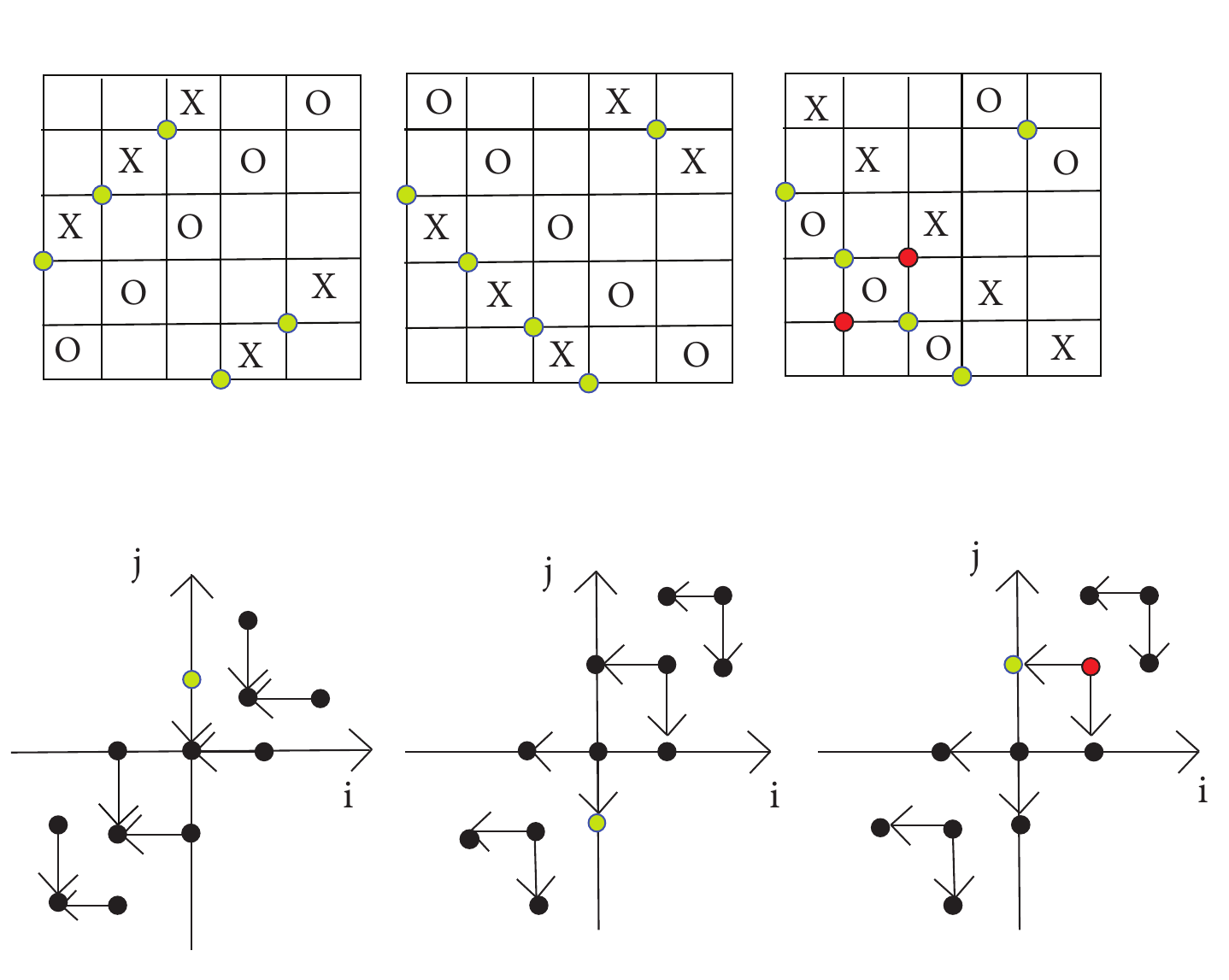}
    \caption{The first picture on the top is a grid diagram for $T_{-2,3}$, while the first picture in the second row shows a part of \textit{reduced} $CFK^{\infty}(T_{-2,3})$. The yellow dots in the grid picture indicate the grid state which is the generator of the $\mathbb{F}[U]$ part of $GH^{-}(T_{-2,3})$. The second and third pictures on top row show the effect of mirroring and interchanging the $\Xmark$- and $\Omark$-markings on the diagram, respectively. In the second row, the yellow dot represents the generator of the free part of $HFK^{-}(T_{-2,3})$ and the next two pictures show the effect of mirroring and orientation reversal of $T_{-2,3}$ on that generator. In the final picture of both rows, the element represented by red dots kills the distinguished homology class.}
    \label{comparative}
\end{figure}

\section{Concordance Properties of $\varepsilon$}

In this section, we prove that $\varepsilon$ is a knot concordance invariant. We make use of the combinatorial theory and some results from there to prove the properties of $\varepsilon$ listed in Theorem \ref{main}. 

Since we also need to consider connected sum of knots using grid diagrams, we define a method of constructing the grid diagram of connected sum of knots. This method allows us to carry the information of the individual knots to the larger diagram of the connected sum.

\begin{figure}[h!]
\centering
    \includegraphics[width=13cm]{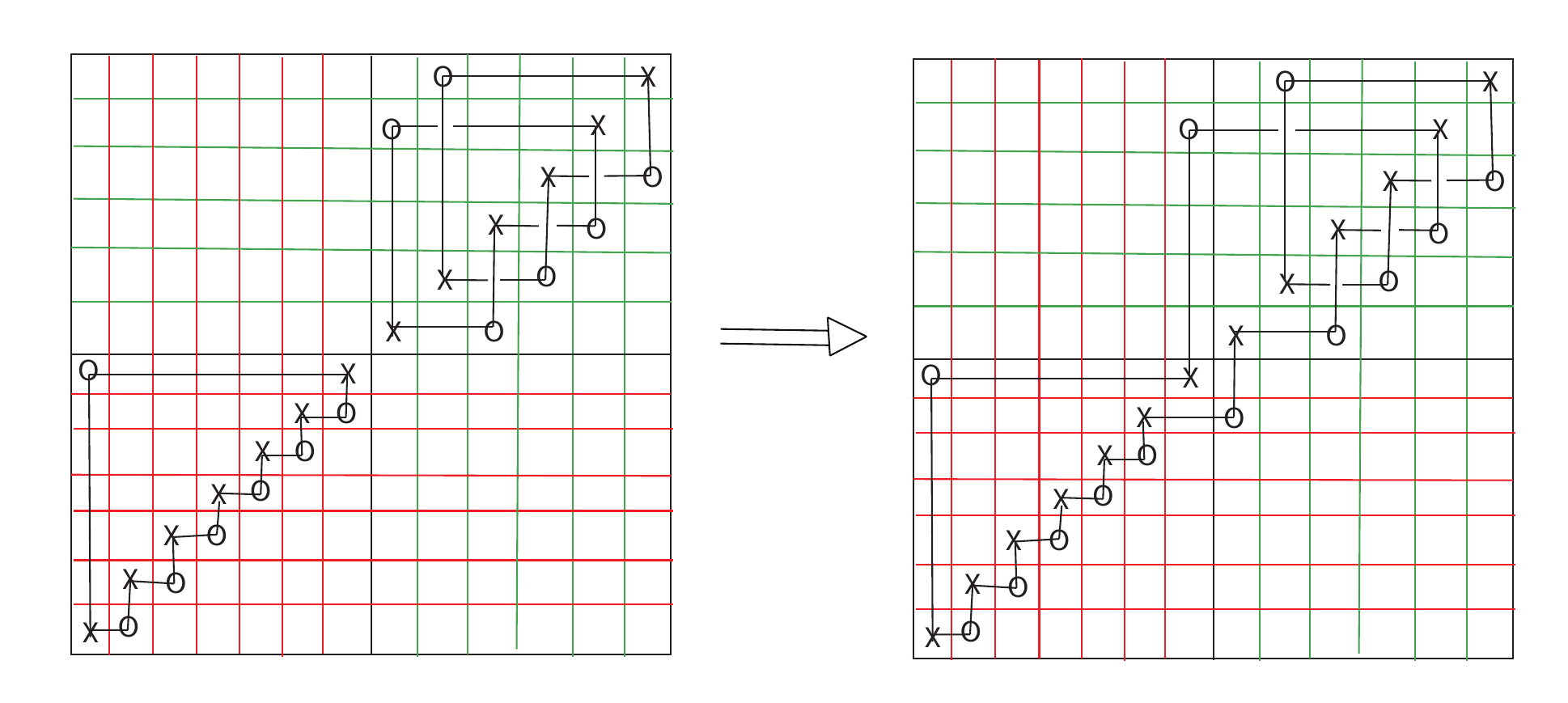}
    \caption{Two schematic diagrams showing how to form the connected sum from the disjoint union, on the left a grid diagram for the disjoint union of a link and the unknot, on the right a grid diagram for their connected sum}
    \label{saddle}
    \end{figure} 

Let $\mathbb{G}_{1}'$ and $\mathbb{G}_{2}'$ be two grid diagrams of the knots $K_1$ and $K_2$ with grid indices $n$ and $m$ respectively. By the definition of grid diagrams, each row and column contain an $\mathbb{X}$-marking. In $\mathbb{G}_1'$, cyclically permute the columns such that the $\mathbb{X}$-marking in the bottom row is in the left corner, and in $\mathbb{G}_2'$, cyclically permute the columns as well such that the $\mathbb{X}$-marking in the top row moves to the right corner. Call the resulting diagrams $\Grid_1$ and $\Grid_2$. After this step, place $\Grid_1$ to upper-right and $\Grid_2$ in the lower-left corner diagonally, and form a new grid diagram with grid index equal to $n+m$. Notice that this gives a grid diagram for $K_1 \sqcup K_2$. To produce a grid diagram for $K_1 \# K_2$, we interchange the $\mathbb{O}$-markings in the first column of $\Grid_1$ and the last column of $\Grid_2$ in the new diagram, cf. Figure [\ref{saddle}], while keeping the $\Xmark$-markings fixed. As discussed in \cite{ozsvath2015grid}, this is a grid realization of a saddle move which performs the connected sum operation. We denote the resulting diagram $\Grid_1 \# \Grid_2$. Furthermore, a grid state in $\Grid_1 \# \Grid_2$ has $n+m$ intersection points. Hence, if $\mathbf{x}_1 \in S(\Grid_1)$ and $\mathbf{x}_2 \in S(\Grid_2)$, we use $\mathbf{x}_1 \cup \mathbf{x}_2$ to denote the union of their coordinates on the resulting grid diagram. They will naturally produce a grid state in $\Grid_1 \# \Grid_2$ since top and right-most edges are omitted on the toroidal grid diagrams when specifying the grid states.

It is a standard result in the literature that for any knot $K$ in $S^3$, $K \# \shortminus K$ is slice and as a result, concordance invariants such as $\varepsilon$ vanish. Another advantage of the connected sum construction using grid diagrams is that we can directly see that $\varepsilon(\Grid \# \shortminus \Grid)=0$ where $\Grid$ is a grid diagram of $K$. As shown schematically in Figure \ref{connectslice}, we start by positioning the grid diagrams of $K$ and $\shortminus K$. To compute $\varepsilon$, we follow our construction and reflect the diagram with respect to a horizontal axis and change the markings. We immediately see that, we obtain a diagram for $K$ in the upper-left and a diagram for $\shortminus K$ in the lower-right of the diagram. By cyclically permuting one of $\Grid$ or $\shortminus \Grid$, we observe that we go back to the original diagram. This means that the homology class $[x]$ that we keep track of to compute $\varepsilon$ is mapped to itself in the final diagram. As a result, it is still the generator of the free part of $GH^-$ in the resulting diagram, which implies that $\varepsilon(\Grid \# \shortminus \Grid)=0$.

\begin{figure}[h!]
\centering
    \includegraphics[width=10cm]{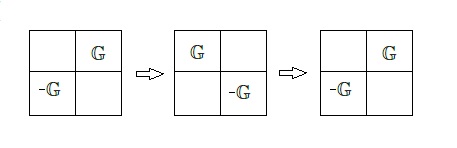}
    \caption{Schematic picture for the connected sum, reflected and cyclically permuted}
    \label{connectslice}
    \end{figure}

Notice that this construction increases the complexity of the chain complex, that is if $\Grid_1$ and $\Grid_2$ has $n!$ and $m!$ generators respectively, the connected sum has $(n+m)!$ generators. However, we have the following lemma which tells us that we can transfer these grid states preserving the additivity of Alexander grading.

\begin{Lemma} \label{3.1}
Given two grid diagrams $\Grid_1$ and $\Grid_2$ for $K_1$ and $K_2$, let $\Grid'=\Grid_1 \sqcup \Grid_2$ and $\Grid= \Grid_1 \# \Grid_2$ denote the grid diagrams of their disjoint union and connected sum constructed using above description. Let $\mathbf{x_1} \in S( \Grid_1)$, $\mathbf{x_2} \in S( \Grid_2)$ and $(\mathbf{x_1} \cup \mathbf{x_2})  \in S( \Grid_1 \# \Grid_2)$. Then $A(\mathbf{x_1} \cup \mathbf{x_2})= A(\mathbf{x_1})+A(\mathbf{x_2})$.
\end{Lemma}
\begin{proof}
First, given the above construction, perform a torus translation on $\Grid'$ so that $\Grid_1$ is positioned on the upper-left and $\Grid_2$ is positioned in the lower-right corner. Since there are no new $\Xmark$ or $\Omark$-markings, or intersection points of grid states to the north-east (NE) of the diagrams $\Grid_1$ and $\Grid_2$ when forming the disjoint union, there will be no shift in the gradings of $\mathbf{x_1}$ and $\mathbf{x_2}$ viewed as a set of coordinates in $\Grid'$. After interchanging the $\Omark$-markings and obtaining $\Grid$, we observe that $\Xmark$-markings and the position of grid states do not change. Hence $\mathcal{J}(\mathbf{x_1} \cup \mathbf{x_2},\mathbf{x_1} \cup \mathbf{x_2})$ and $\mathcal{J}(\Xmark, \Xmark)$ does not change calculated in $\Grid$ or $\Grid'$. The only additional contribution comes from the relative position of $\Omark$-markings, which yields $\mathcal{J}(\Omark',\Omark')+1=\mathcal{J}(\Omark,\Omark)$. This immediately tells us $M_{\Omark}(\mathbf{x_1} \cup \mathbf{x_2})=M_{\Omark'}(\mathbf{x_1} \cup \mathbf{x_2})+1$. Notice that $\Grid'$ represents a link, so we use the normalized Alexander function for $\Grid'$ from Equation (\ref{alexandergridlink}).

In this case $l=2$ and combining this with the Maslov grading shift above yields the result.
\end{proof}

To show that $\varepsilon$ defined using grid diagrams is a concordance invariant, we need to prove some properties of $\varepsilon$, analogous to that of $\varepsilon$ in knot Floer homology. For that purpose, we will be using the following results from \cite{ozsvath2015grid}.    

\begin{Prop}[\cite{ozsvath2015grid} Proposition 2.6.11]
\label{3.2}
Suppose that two knots $K_1,K_2$ can be connected by a smooth, oriented, genus $g$ cobordism $W$. Let $\mathcal{U}_b(K_1)$ and $\mathcal{U}_d(K_2)$ denote the disjoint unions of $K_1$ and $K_2$ with $b$ and $d$-many unlinked, unknotted components respectively. Then there are knots $K_1', K_2'$ and integers $b,d$ with the following properties:
\begin{itemize}
    \item $\mathcal{U}_b(K_1)$ can be obtained from $K_1'$ by $b$ simultaneous saddle moves.
    \item $K_1'$ and $K_2'$ can be connected by a sequence of $2g$ saddle moves.
    \item $\mathcal{U}_d(K_2)$ can be obtained from $K_2'$ by $d$ simultaneous saddle moves.
\end{itemize}
\end{Prop}

\begin{Lemma}[\cite{ozsvath2015grid} Lemma 8.4.2]
\label{3.3}
Let $L$ be an oriented link, and let $L' = \mathcal{U}_1(L)$. Then there is an isomorphism of bigraded $\mathbb{F}[U]$-module
\begin{align*}
    cGH^{-}(L') \cong cGH^{-}(L)[[1,0]] \oplus cGH^{-}(L)
\end{align*}
\end{Lemma}

\begin{Prop}[\cite{ozsvath2015grid} Proposition 8.3.1]
\label{3.4}
Let $W =\mathbb{F}_{(0,0)} \oplus \mathbb{F}_{(-1,-1)} $. If $L'$ is obtained from $L$ by a split move, then there are $\mathbb{F}[U]$-module maps
\begin{align*}
   \sigma: cGH^{-}(L) \otimes W \rightarrow cGH^{-}(L') \\
    \mu: cGH^{-}(L') \rightarrow cGH^{-}(L) \otimes W
    \end{align*}
    
 with the following properties
    \begin{itemize}
        \item $\sigma$ is homogenous of degree $(-1,0)$,
        \item $\mu$ is homogenous of degree $(-1,-1)$,
        \item $\mu \circ \sigma$ is multiplication by $U$,
        \item $\sigma \circ \mu$ is multiplication by $U$.
    \end{itemize}    
\end{Prop}

We can now prove Theorem \ref{main}. 
\begin{proof}[Proof of Theorem \ref{main}]
First, we argue for part a) of Theorem 1.1. By Lemma \ref{3.3}, $cGH^{-}(\mathcal{U}_b(K))/Tors \cong (cGH^{-}(K)/Tors)^{2^b}$. Recall that for $K$ a knot,
 \[cGH^{-}(K) \cong GH^{-}(K) \cong \mathbb{F}[U] \oplus Tors\]
 Let $\alpha$ be the generator of the $\mathbb{F}[U]$-module.
Let $\Grid$ be a diagram for the knot $K$. Now if $K$ is slice, $-K$ is also slice, i.e. both $K$ and $-K$ are concordant to the unknot. By Proposition \ref{3.2}, there exists $K'$ such that $\mathcal{U}_b(K)$ can be obtained from $K'$ by $b$ simultaneous saddle moves. We arrange the saddle moves such that we perform them one at a time. Now by the Proposition \ref{3.4} and by the injectivity of the maps $\sigma$ and $\mu$, the homology class represented by $\alpha$ is non-zero in $cGH^{-}(\mathcal{U}_b(-K))/Tors$. In this case, we can observe that the grid state $\mathbf{x_{\tau}}$ in $S(\mathbb{G})$ representing the homology class of $\alpha$ maps to a non-torsion homology class in $GH^{-}(-\mathbb{G})$. This is the same homology class represented by the grid state that we obtain by reflecting $\mathbb{G}$ and interchanging $\Xmark$ and $\Omark$ markings. This can be seen as follows: starting with a grid diagram of $K$, one can obtain a grid diagram representing the unknot through the moves described in Proposition \ref{3.2}. Reflect this unknot diagram and interchange $\Xmark, \Omark$ markings and we get the grid diagram of the orientation reversed unknot. Following this, we repeat the moves listed above from $K$ to unknot but now in the reverse order, to obtain the grid diagram $-\mathbb{G}$ of $-K$ (this can be done since $-K$ is also concordant to unknot). In every step, $\alpha$ is mapped to a non-torsion element in the homology of the resulting grid diagram and also $[x_{\tau}]$ is mapped to the same element that is obtained by reflecting $\mathbb{G}$ and interchanging $\Xmark,\Omark$ markings.

Hence, $\alpha$ is a non-zero element in $cGH^{-}(-K)/Tors$. This implies that $\varepsilon(K) = 0$

\vspace{0.2cm}

To prove part b), we rely on the symmetry properties of the grid homology and the dual complex (which is the grid complex for $-K$) constructed in \cite{ozsvath2015grid} (Proposition 7.4.3). The explicit isomorphism in this proposition maps all grid states of $\mathbb{G}$ to grid states in $\shortminus \mathbb{G}$. In particular, the non-torsion part of $GH^-(\mathbb{G})$ is supported at bigrading $(- 2 \tau, - \tau)$ and the non-torsion part of $GH^-(\shortminus \mathbb{G} )$ is supported at $(2 \tau , \tau)$. As a result, we have $GH^-(\shortminus \mathbb{G}) \cong GH^-(\mathbb{G})^* $ which implies that $\varepsilon( \shortminus \mathbb{G})= \shortminus \varepsilon (\mathbb{G})$.

\end{proof}

To prove c) and d) in Theorem 1.1, first we prove the following lemma. 

\begin{Lemma}
\label{3.5}
Let $L_1,L_2$ be two knots in $S^3$ and $[\mathbf{x}_{\tau_i}]$ be the generators of \\ $cGH^{-}(L_i)$/Tors $ \cong \mathbb{F}[U]$ for $i=1,2$. Then $[\mathbf{x}_{\tau_1} \cup \mathbf{x}_{\tau_2}]$ is a non-zero element in
\\ $cGH^{-}(L_1 \# L_2)/Tors$. 
\end{Lemma}

Before starting the proof, we describe the union `$\cup$' notation here. Let $\mathbf{x}$ and $\mathbf{y}$ be two chains in $GC^-(\Grid_1)$ and $GC^-(\Grid_2)$ respectively, and assume that we can write them as $\mathbf{x} = \mathbf{x_1} + \mathbf{x_2}$ and $\mathbf{y} = \mathbf{y_1} + \mathbf{y_2}$ where $\mathbf{x}_i, \mathbf{y}_i$ are grid states in $\mathbb{G}_1, \mathbb{G}_2$. Then the union notation is used as follow;
\begin{align*}
    \mathbf{x} \cup \mathbf{y} = \mathbf{x}_1 \cup \mathbf{y}_1 + \mathbf{x}_1 \cup \mathbf{y}_2 + \mathbf{x}_2 \cup \mathbf{y}_1 + \mathbf{x}_2 \cup \mathbf{y}_2
\end{align*}
where the union in each one of the summands is as described in Lemma \ref{3.1}. 

First, we observe that if $\alpha$ is the generator of the $\mathbb{F}[U]$ module $GH^{-}(L_1 \# L_2)/Tors$, then we can argue that $[\mathbf{x}_{\tau_1} \cup \mathbf{x}_{\tau_2}] = p(U) \cdot \alpha$, for some non-zero polynomial $p(U) \in \mathbb{F}[U]$. Assuming the additivity of $\tau$ invariant under connected sum operations (from literature in knot Floer homology, see \cite[Section 7]{ozsvath2004holomorphic}) i.e. \[\tau(K_1 \# K_2) = \tau(K_1) + \tau(K_2)\] and Lemma \ref{3.1}, we see that $p(U) = 1$ and $[\mathbf{x}_{\tau_1} \cup \mathbf{x}_{\tau_2}]$ is indeed a generator of the non-torsion part of $GH^{-}(L_1 \# L_2)$.
Hence, to determine the $\varepsilon$ of $L_1\#L_2$, we need to track the element $[\mathbf{x}_{\tau_1} \cup \mathbf{x}_{\tau_2}]$ after reflecting the grid diagram of $L_1\#L_2$ with respect to a horizontal axis and switching all $\mathbb{X}$ and $\mathbb{O}$ markings.

Notice that proving Lemma \ref{3.5} for $L_1 \sqcup L_2$ and combining it with Proposition \ref{3.4}, the result will follow for $L_1 \# L_2$.  

\begin{figure}
    \centering
    \includegraphics[width=6cm]{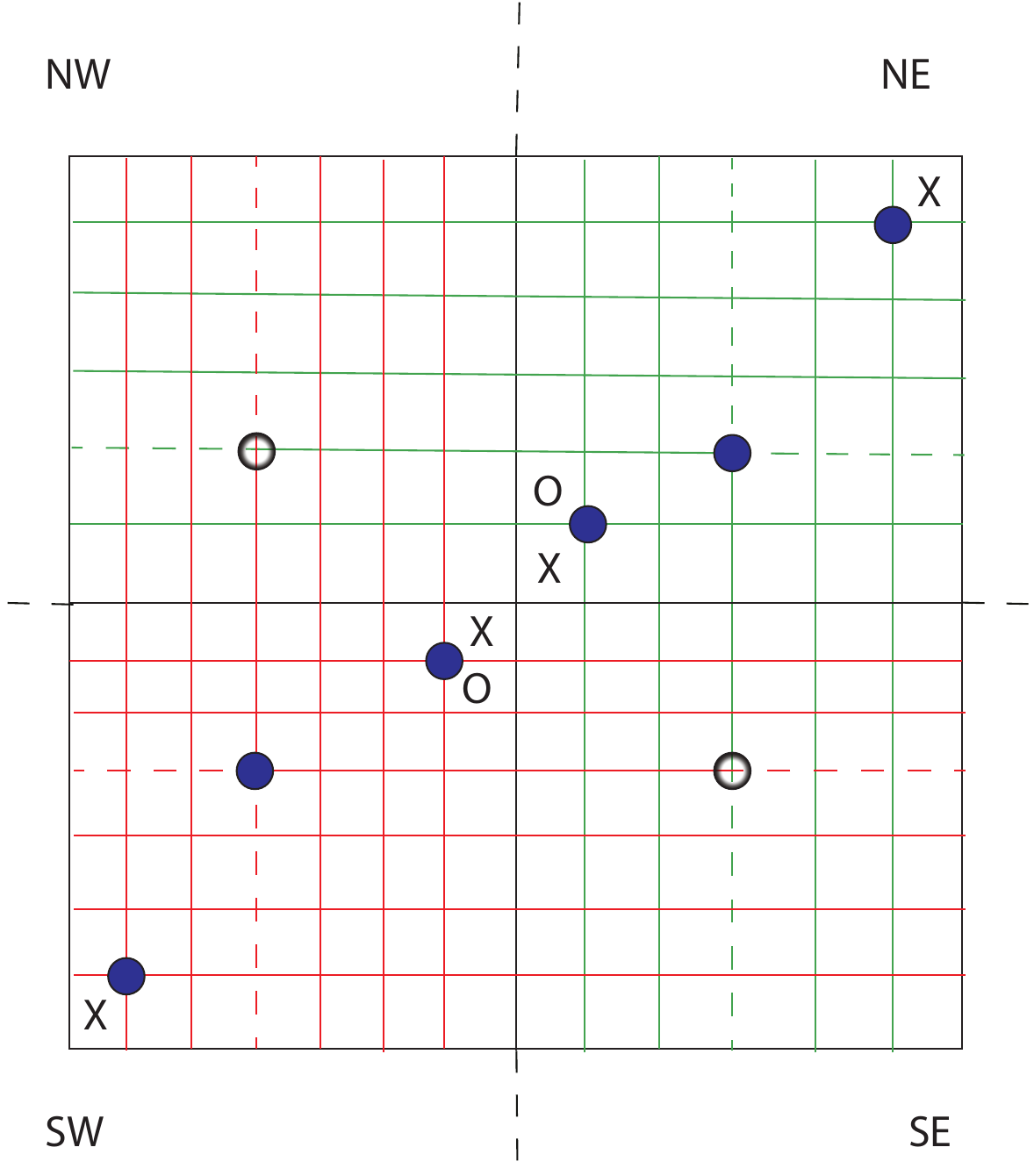}
    \caption{Rectangular regions marked by dashed lines with vertices $\mathbf{y}$ and $\mathbf{x'}$}
    \label{gr1}
\end{figure} 

\begin{figure}
    \centering
    \includegraphics[width=6cm]{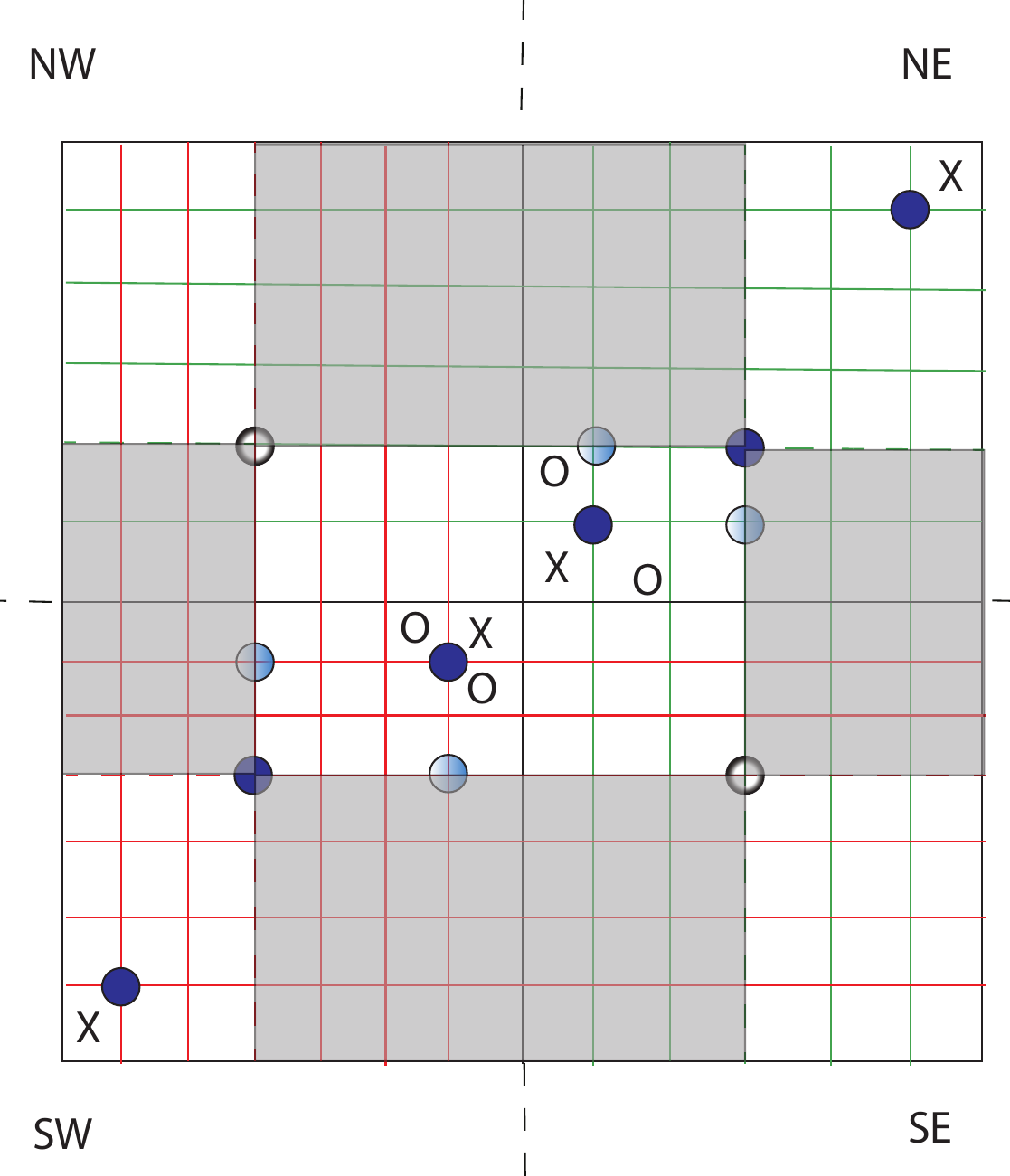}
    \caption{$\mathbf{x'}, \mathbf{y}$ and grid representatives of $\mathbf{y}$ when $\mathbf{x'}$ and $\mathbf{y}$ are connected via shaded rectangles}
    \label{gr2}
\end{figure}

\begin{proof} [Proof of Lemma \ref{3.5}]
First we show that the element $[\mathbf{x}_{\tau_1} \cup \mathbf{x}_{\tau_2}]$ is non-zero in the homology of $GC^{-}(\mathbb{G}')$, where $\mathbb{G}'$ denotes the grid diagram of $L_1 \sqcup L_2$.

Observe that for a grid state $\mathbf{x} \in S(\mathbb{G}')$ which has intersection points lying entirely in the green (upper-right) and the red portion (lower-left) (See Figure [\ref{gr1}]), grid states in $\partial^{-}_{\Xmark}(\mathbf{x})$ can be of two kind. In the first case, these grid states in the boundary have components only in the green or the red portion of the grid diagram which means that the rectangles between these grid states are rectangles coming from $\Grid_1$ and $\Grid_2$ prior to forming disjoint union. If $\partial^{-}_{\Xmark}(\mathbf{x})$ consists of only these type of grid states, then we have no new contribution in the homology. Hence, the triviality/non-triviality of $\mathbf{x}$ in $\Grid'$ is determined by the individual rectangles counted in $\Grid_1$ and $\Grid_2$ by $(\partial^-_{\Xmark})_1$ and $(\partial^-_{\Xmark})_2$ (where $(\partial^-_{\Xmark})_i$ denote the boundary map of $GC^-(\Grid_i)$). In the second case,  they can differ from $\mathbf{x}$ in two intersection points, one of which lie in NW and other one lie in SE grid blocks. Figure [\ref{gr1}] and Figure [\ref{gr2}] show such examples.

Non-triviality of $[\mathbf{x}_{\tau_1}]$ and $[\mathbf{x}_{\tau_2}]$ in $GH^{-}(L_1)$ and $GH^{-}(L_2)$, respectively, combined with the following construction will give us the necessary result. We start with performing one $X:NE$ and one $X:SW$ stabilization in both $\Grid_1$ and $\Grid_2$ at the upper-right and lower-left part of the diagrams. New diagrams contain two extra intersection points coming from the intersection points of the newly introduced vertical and horizontal lines. After a stabilization move, this new intersection point is added to all grid states which creates a bijection between the grid states before the stabilization and the grid states after the stabilization. As a result, these two new intersection points are added to $\mathbf{x}_{\tau_1} \cup \mathbf{x}_{\tau_2}$ after stabilizing the diagrams as shown in Figure \ref{gr1}.

Now we start our proof for the non-triviality of  $[\mathbf{x}_{\tau_1} \cup \mathbf{x}_{\tau_2}] \in cGH^-(\mathbb{G}')$. Let $\mathbf{x}', \mathbf{y} \in S(\mathbb{G}')$ such that $\mathbf{x}'$ and $\mathbf{y}$ differ in two components in NE and SW part of the grid (ref. Figure [\ref{gr1}], where $\mathbf{x'}$ contains the dark blue dots and $\mathbf{y}$ differs from $\mathbf{x'}$ hollow grey dots in NW and SE part of the grid, and they agree in all other coordinates). Assume that $\mathbf{x'}$ is a grid state in the chain $\mathbf{x}_{\tau_1} \cup \mathbf{x}_{\tau_2}$.

 As discussed earlier, given a grid state $\mathbf{x_1} \in S(\mathbb{G}_1)$ and another grid state $\mathbf{x_1'} \in \partial^-_1(\mathbf{x_1})$, any empty rectangle connecting $\mathbf{x_1}$ and $\mathbf{x_1'}$ can be viewed as an empty rectangle in $\Grid'$ with the same $\mathbb{O}$-multiplicity by simply taking the union of both grid states with another grid state $\mathbf{x_2} \in S(\mathbb{G}_2)$. Reversing this idea, we decompose the domains in $\Grid'$ into smaller domains in $\Grid_1$ and $\Grid_2$ by finding corresponding \textit{grid representatives} in $\Grid_1$ and $\Grid_2$. Hence, let $\mathbf{x}$ be a grid state in the cycle that represents the homology class $[\mathbf{x}_{\tau_1}\cup \mathbf{x}_{\tau_2}]$ and let $\mathbf{y}$ be another grid state such that there is a rectangle connecting them such that $\mathbf{x} = \partial^-_{\Xmark} \mathbf{y}$. We find the grid representatives of $\mathbf{y}$ in $\mathbb{G}_1$ and $\mathbb{G}_2$ such that:
   
\begin{itemize}
    \item $\mathbf{y_i} \in S(\mathbb{G}_i)$ for $i=1,2$.
    \item $\mathbf{y_i}$ differs from $\mathbf{x}|_{\mathbb{G}_i}$ in two coordinates one of which is the newly introduced intersection point that is closer to the middle of the diagram where $\mathbb{G}_1, \mathbb{G}_2$ meet. Notice that this also determines the other coordinate where these grid states differ.
    \item the sum of the $\mathbb{O}$-multiplicities in the empty rectangles connecting $\mathbf{x}|_{\mathbb{G}_1}, \mathbf{y_1}$ and $\mathbf{x}|_{\mathbb{G}_2},  \mathbf{y_2}$ is 1 more than the $\mathbb{O}$-multiplicities of the empty rectangle connecting $\mathbf{x}, \mathbf{y}$ in $\mathbb{G}'$.  
    \end{itemize}
    
In Figure [\ref{gr1}], such grid representatives of $\mathbf{y}$ are shown. After finding these grid representatives, the above points imply;
\[\mathbf{x} = \partial_{\Xmark}^- \mathbf{y} \Rightarrow  \mathbf{x}|_{\mathbb{G}_i} = (\partial_{\Xmark}^{-})_i \mathbf{y_i}, i=1,2\]
Since $\mathbf{x}|_{\Grid_i}$ are grid states in the chains $\mathbf{x}_{\tau_i}$ and $\mathbf{x}_{\tau_i} \notin Im(\partial_{\Xmark}^-)_i$ for $i=1,2$, we get that $(\mathbf{x}_{\tau_1} \cup \mathbf{x}_{\tau_2}) \notin Im(\partial_{\Xmark}^-)$. 

By construction, $\partial_{\Xmark}^- (x_{\tau_1} \cup x_{\tau_2}) = 0$. This is because if there is some $\mathbf{y}$ that differs from $\mathbf{x}'$ in two vertices (where $\mathbf{x}'$ is a grid state in the chain $(x_{\tau_1} \cup x_{\tau_2}))$, then both the rectangles with those vertices will have $\Xmark$-markings in it and they will not be counted by the differential $\partial^-_{\Xmark}$. (see Figure [\ref{gr1}]).

To complete the proof, we need to show that [$\mathbf{x}_{\tau_1} \cup \mathbf{x}_{\tau_2}$] is a non-torsion element in $GH^{-}(\mathbb{G}')$. If [$\mathbf{x}_{\tau_1} \cup \mathbf{x}_{\tau_2}$] is a torsion element, then for all $n > N$, $U^n \cdot [\mathbf{x}_{\tau_1} \cup \mathbf{x}_{\tau_2}] = 0$, for some natural number $N$. This implies that there exists $\mathbf{y'} \in GC^{-}(\mathbb{G}')$ such that $\partial_{\Xmark}^- \mathbf{y'} = V_i^n \cdot (\mathbf{x}_{\tau_1} \cup \mathbf{x}_{\tau_2})$, for some $i$. Fpr this case, we can follow the same strategy as before and decompose the domains into $\Grid_1$ and $\Grid_2$, for which appropriate $\Omark$-multiplicities add up to produce $U^n$. Using the fact that $[\mathbf{x}_{\tau_i}]$ is non-torsion in $GH^{-}(\mathbb{G}_i)$ for $i=1,2$, the previous argument completes the proof.

\end{proof}

Now, we are ready to prove the behavior of $\varepsilon$ under connected sum operation.

\begin{Lemma} Let $\Grid_1$ and $\Grid_2$ be grid diagrams for the knots $K_1$ and $K_2$ respectively.
If $\varepsilon(\Grid_1) = \varepsilon(\Grid_2)$, then $\varepsilon(\Grid_1 \# \Grid_2) = \varepsilon(\Grid_1)$ and 
if $\varepsilon(\Grid_2) = 0$, then $\varepsilon(\Grid_1\# \Grid_2) = \varepsilon(\Grid_1)$.
\end{Lemma}

\begin{proof}
First, let $\varepsilon(\Grid_1) = \varepsilon(\Grid_2) = 0$ and $[\mathbf{x}_{\tau_i}]$ be the generator of $GH^{-}(\Grid_i)/Tors$ for $i=1,2$. 
Then by definition, we know that $[\mathbf{x}_{\tau_i}]$ is non-zero in $GH^{-}(\Grid_i)/Tors$ and in $GH^{-}(\shortminus \Grid_i)/Tors$ $i=1,2$. As a result, by Lemma \ref{3.5}, $[\mathbf{x}_{\tau_1} \cup \mathbf{x}_{\tau_2}]$ is non-zero in both $cGH^{-}(\Grid_1 \sqcup \Grid_2)/Tors$ and $\cgh(\shortminus (\Grid_1 \sqcup \Grid_2))$. Hence, $\varepsilon(\Grid_1 \# \Grid_2) = 0$ as $\mu$ maps non-torsion elements of $cGH^{-}(\Grid_1 \sqcup \Grid_2)$ to non-torsion elements of $cGH^{-}(\Grid_1 \# \Grid_2)$.

Now assume $\varepsilon(\Grid_1) = \varepsilon(\Grid_2) = 1$. By the definition of $\varepsilon$, we have $\partial_{\Xmark}^- ([\mathbf{x}_{\tau_i}]) \neq 0$ for $i=1,2$. To compute $\varepsilon(\Grid_1 \# \Grid_2)$, we look at $[\mathbf{x}_{\tau_1} \cup \mathbf{x}_{\tau_2}]$ in the grid complex of $\Grid_1 \# \Grid_2$ after reflecting the diagram and changing $\mathbb{X}$ and $\mathbb{O}$ markings as before. Observe that in this complex, we have the following relation;

\begin{align*}
    ((\partial_{\Xmark}^{-})_1([\mathbf{x}_{\tau_1}]) \cup [\mathbf{x}_{\tau_2}]) + (\mathbf{x}_{\tau_1} \cup (\partial_{\Xmark}^{-})_2([\mathbf{x}_{\tau_2}])) \in \partial_{\Xmark}^- ([\mathbf{x}_{\tau_1} \cup \mathbf{x}_{\tau_2}])
\end{align*}
which means that we see a non-trivial contribution to the boundary map and $\partial_{\Xmark}^- [\mathbf{x}_{\tau_1} \cup \mathbf{x}_{\tau_2}] \neq 0$. Hence, we conclude that $\varepsilon(\Grid_1 \# \Grid_2) = 1$. In the case $\varepsilon(\Grid_1) = \varepsilon(\Grid_2) = -1$, we apply the same argument but start with the diagrams of $-K_1$ and $-K_2$.

The cases $\varepsilon(\Grid_1)= \pm 1$ and $\varepsilon(\Grid_2)=0$ follows from the same technique. In short, we decompose the domain in $\Grid_1 \# \Grid_2$ to corresponding the grid representatives. One can see that the contribution of the domain coming from $\Grid_1$ determines $\varepsilon$ in this case and this completes the proof.

\end{proof}

\section{Some Computations for Torus Knots and Iterated Cables of Torus Knots}

In this section, we compute $\varepsilon$ for negative torus knots and examine the behavior of $\varepsilon$ under cabling.

\begin{proof}[Proof of Theorem \ref{torus}] \vspace{0.2cm}

Before we look at the different cases, it is important to note that in Section 6.3 of \cite{ozsvath2015grid}, it is shown that, the element denoted as $\mathbf{x}^+$ which takes the intersection points of upper-right corners of each $\Xmark$-marking, is a non-torsion cycle with maximal Alexander grading when we consider $(-p,q)$ torus knots. In other words, this is the generator that gives the $\tau$ of $(-p,q)$ torus knots. Hence, we will keep track of this element when we carry on the computations for $\varepsilon(\Grid_{-p,q})$. \vspace{0.1cm}

First, we would like to show that $\varepsilon(\Grid_{p,q})=0$ if $q=1$. In this case, grid index is $p+1$ and grid diagram is encoded with following permutations $\sigma_{\Omark}=(1, 2, 3, \dots , p+1)$ and $\sigma_{\Xmark}=(p+1, 1, 2 \dots , p)$. Figure [\ref{qis1}] shows a schematic picture of this case.

\begin{figure}[h!]
\centering
    \includegraphics[width=0.32\textwidth]{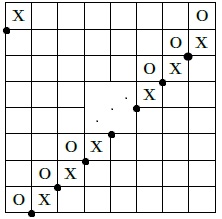}
    \caption{A schematic diagram for the case $q=1$ and black dots represent $\mathbf{x}^+$ }
    \label{qis1}
\end{figure}

As we can see in the diagram, such a diagram will always consist of $\Omark$-markings on the diagonal followed by $\Xmark$-markings in the diagonal below with a single $\Xmark$-marking at the top-left corner. Now if we reflect this diagram with respect to a horizontal axis and change all $\Xmark$ and $\Omark$-markings, we will obtain a diagram as in Figure [\ref{reflectedq1}];

\begin{figure}[h]
\centering
    \includegraphics[width=0.32\textwidth]{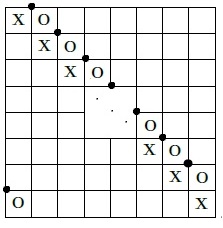}
    \caption{Grid diagram for $\shortminus \Grid_{-p,1}$}
    \label{reflectedq1}
\end{figure} 

By looking at the diagram of $\shortminus \Grid_{-p,1}$, we can immediately see that the element we obtain is still $\mathbf{x}^+$ and independent of the diagram, this element is always a cycle, referred as \textit{canonical grid state} in \cite{ozsvath2015grid}. Hence, the element we obtain is non-trivial in $GH^-(\shortminus \Grid_{-p,1})$ which implies that $\varepsilon(\Grid_{p,1})=0$.

\vspace{0.2cm}

Now assume that $q>1$. We would like to show that $\varepsilon(\Grid_{-p,q})=-1$. We will keep track of the same element as in the previous discussion, but the diagram will be slightly different. Notice that starting with $\mathbf{x}^+$, after reflecting the diagram and changing the markings, this element becomes $\mathbf{o}_+$ which denotes the grid state that consists of the lower-right intersection points of $\Omark$-markings. In this case, we see that for any other grid state $\mathbf{y} \in S(\shortminus \Grid_{-p,q})$ such that $\mathbf{o}_+$ and $\mathbf{y}$ differ in two consecutive coordinates, there exists exactly one $1 \times 1$ rectangle with an O-marking that goes from $\mathbf{y}$ to $\mathbf{o}_+$, and one can see that for such a grid state $\mathbf{y}$, there are no other elements in the boundary due to the position of the intersection points, either the rectangles cancel in pairs or they contain $\Xmark$-markings (see the Figure [\ref{qbigger1}]). Also the position of the $\Xmark$- and $\Omark$-markings make sure that $\mathbf{o}_+$ is a non-zero element in the homology of $-T_{-p,q}$.  Hence, this element is in the image of the map $\partial^{horz}_{\Xmark}$ which implies that $\varepsilon(\Grid_{-p,q})=-1$ for $q>1$.

\begin{figure}[h!]
\centering
    \includegraphics[width=0.75\textwidth]{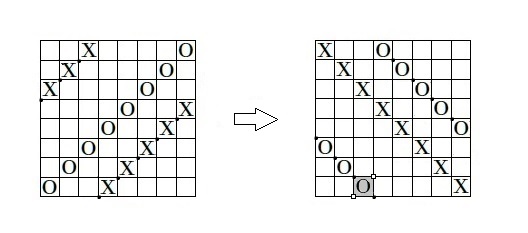}
    \caption{An illustration for $\Grid_{-3,5}$, on the right black dots are the elements we keep track of, and small squares represent another element that is in the boundary of the distinguished element, and these grid states agree on all other coordinates. Shaded $1 \times 1$ square shows an example of such a $V_i$ contribution to the boundary map for some $i$ }
     \label{qbigger1}
    \end{figure} 
Final case is when $q<1$, however we appeal to the fact that $\Grid_{p,-q}$ is same as (not considering the orientation reversal) $\shortminus \Grid_{p,q}$. Hence, $\varepsilon(\Grid_{p,q})=1$ for $q<1$, following from the previous case and part b) of Theorem \ref{main}.
\end{proof}

    
We would like to continue this discussion by proving the behaviour of $\varepsilon$ under $(p,q)$-cabling of negative torus knots.

\begin{proof}[Proof of Theorem \ref{cable}]

Before looking at the different cases of $\varepsilon$ and how it behaves under cables, we would like to describe here how we obtain $(r, rn-1)$-cable of a negative torus knot using grid diagrams, where $n$ is the writhe of $\mathbb{G}_{-p,q}$ and $r$ is the cabling coefficient. After understanding this, we follow a similar path as in \cite{hom2014bordered} and refer to Van Cott's work \cite{vancott} to generalize it to $(p,q)$-cables. Throughout the proof, we use $\mathbb{G}$ to denote the standard grid diagram for the negative torus knot $T_{-p,q}$. We begin by explaining how we obtain the $(r, rn-1)$ from a grid. 

In any grid diagram $\mathbb{G}$, an $\Xmark$ or $\Omark$-marking is a corner and it can be one of the four types; NW (northwest), NE (northeast), SW (southwest) or SE (southeast). If an $O$-marking is a SE corner, we use the notation $o_{SE}$, and a similar notation for other corners and markings. In the specific case that we have the standard grid diagram for a negative torus knot $\mathbb{G}_{-p,q}$, we have exactly $p$-many $o_{SW}$, $p$-many $o_{NE}$, $(q-p)$-many $o_{SE}$ corners. Similarly, we have $p$-many $x_{SE}$ and $q$-many $x_{NW}$ corners. To obtain a $(r,rn-1)$-cable, we replace empty squares with $r \times r$ empty blocks and for marked squares, we replace them with $r \times r$ blocks with the same marking, also repeating the same corner of the marking. This will give us $r$ parallel copies of $\mathbb{G}_{-p,q}$. To introduce one full negative twist, we choose the $o_{SE}$ $r \times r$ block and introduce the twist as shown in Figure [\ref{cabler3}] and Figure  [\ref{cabler3twist}]. The resulting diagram will have $r^2n + (r-1)$ negative crossings where $n$ is the writhe of the original knot As pointed out in \cite{heddenthesis}, this is the total number of crossings we should have after cabling.

\begin{figure}[h!]
\centering
    \includegraphics[width=14cm]{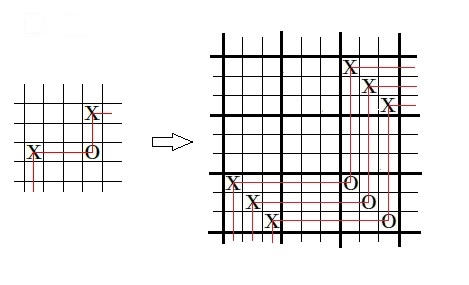}
    \caption{A local picture on a grid diagram where $r=3$ and we repeat all the marked corners preserving their directions}
    \label{cabler3}
\end{figure} 

\begin{figure}[h!]
\centering
    \includegraphics[width=12cm]{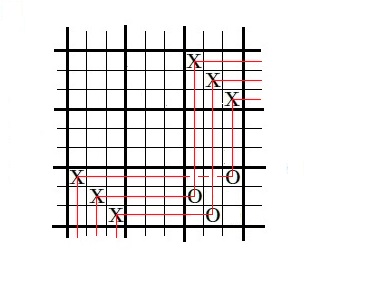}
    \caption{We introduce a full negative twist by changing the $O$-markings in $o_{SE}$ corner. This operation produces $(r-1)$-many new negative crossings. }
    \label{cabler3twist}
\end{figure} 

\begin{enumerate}
    \item Assume now $\varepsilon(\mathbb{G})=0$. By Theorem \ref{torus}, this is only possible when $|q|=1$ and this is the case when we have the unknot. Hence, if we apply the above cabling procedure, we will obtain a torus knot. This immediately tells us that $\varepsilon(\mathbb{K}_{r,rn-1})= \varepsilon(\mathbb{G}_{r, rn-1})$ where $r$ is the cabling coefficient and $n$ is the writhe of the original knot.
    
    \item More interesting case is when $\varepsilon(\mathbb{G}) \neq 0$. In this case, we start with a non-trivial negative torus knot. Assume that $q>1$, so that $\varepsilon(\mathbb{G})=-1$. Recall that, as discussed in the proof of Theorem \ref{torus}, the distinguished cycle $\mathbf{x^+}$ is the element we would like to keep track of. After applying the above cabling operation, $\mathbf{x^+}$ in the resulting diagram still has the maximal Alexander grading and it is a cycle. 
    Using Proposition 6.4.8 in \cite{ozsvath2015grid} we get that this element is non-torsion as well. Hence, we would like to keep track of this element to compute the value of epsilon. Following our description, we reflect this diagram with respect to a horizontal axis and change $X$ and $O$-markings. As before, this grid state is now in the lower-right corners of $O$-markings that we denoted as $\mathbf{o_{+}}$. Now we track this distinguished element in the grid diagram of the $(r,+1)$ cable of $n$-framed $T_{-p,q}$ to find $\varepsilon((T_{-p,q})_{r,rn+1})$.
    
    Observe that we can find the $(r,1)$-cable of an $n$-framed $T_{-p,q}$, reflect that and interchange $\Xmark, \Omark$-markings or equivalently, we can find $(r,1)$-cable of an $n$-framed $T_{p,q}$. Moreover we can perform the twist in the cable at an $o_{SW}$ corner, so that obtain a $(r, rn+1)$-cable. Now we see that around each of the $r$-clusters of $\Omark$-markings, we can perform successive grid commutation moves so that the $r$-clusters of $\Omark$-marking lie anti-diagonally (see Figure [\ref{commutation1}] and Figure [\ref{commutation2}]). It is now easy to see that the boundary of $\mathbf{o_{+}}$ is zero and we can find another grid state the boundary of which contains only $V_i \cdot \mathbf{o_{+}}$ for some $i$, see the picture on the right side of the Figure [\ref{commutation2}]. The boundary of the element only contains $V_i \cdot \mathbf{o_{+}}$ since the boundary of $\mathbf{o_{+}}$ is zero. As a result, $\varepsilon(T_{-p,q})_{r,rn+1} = -1 = \varepsilon(T_{-p.q})$. The case when $q<1$ follows from a similar argument and completes the proof.
\end{enumerate}

\vspace{0.1cm}

\end{proof}    
 
\begin{Rem}
  It is important to note here that a similar technique can be applied for iterated cables of torus knots, since in the grid diagram obtained after cabling the original torus knot, $\mathbf{x^{+}}$ still gives us the non-zero, non-torsion element with the `correct' Alexander grading i.e. one can track $\mathbf{x^{+}}$ (and thus $\mathbf{o_{+}}$) to find $\varepsilon$ of the iterated cable and the argument for those cases hold similarly. 
\end{Rem}

   \begin{figure}[h!]
\centering
    \includegraphics[width=12cm]{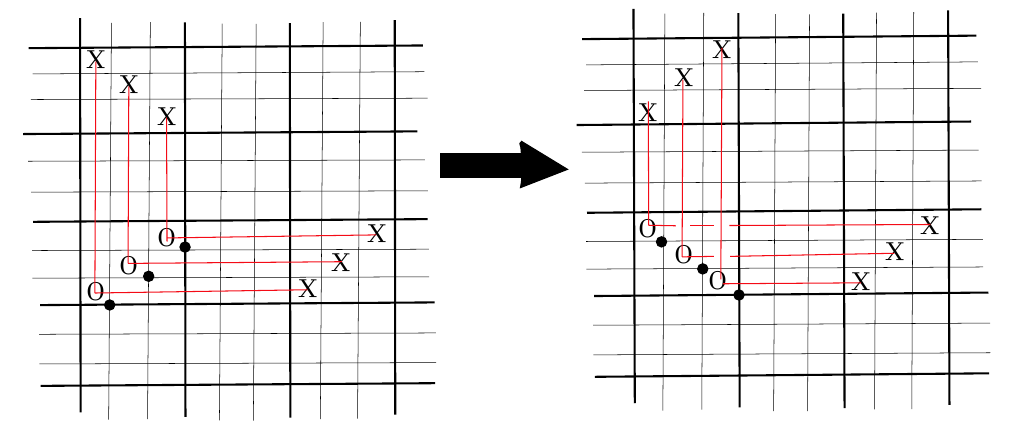}
    \caption{Commutation moves along one of non-twisting $\Omark$-marking cluster}
    \label{commutation1}
\end{figure} 

\begin{figure}[h!]
\centering
    \includegraphics[width=12cm]{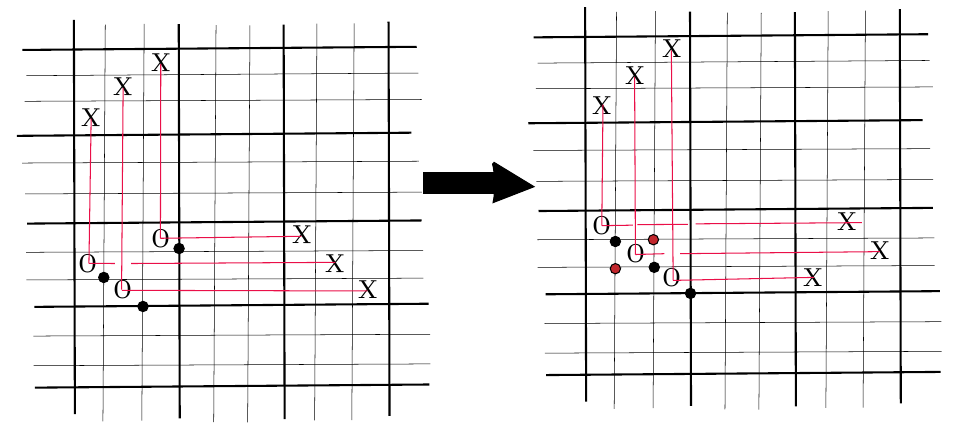}
    \caption{Commutation moves along the twisting $\Omark$-marking cluster and a grid state killing $U \cdot \mathbf{o}_{+}$}
    \label{commutation2}
\end{figure}

\section{Grids and Braids}

In this section, we recall the relation between grids and braids. We follow the notation convention of Ng and Thurston from \cite{gridbraid}.

Given a grid diagram $\mathbb{G}$, we can view it as a braid. We still require that vertical strands go over horizontal strands. We also add the condition that all horizontal segments are oriented from left to right on the diagram. To achieve this, given a grid diagram $\mathbb{G}$, if an $O$ marking on a row is to the left of an $X$ marking, join them as usual inside the grid. If an $O$ marking is to the left of an $X$ marking, then draw a line segment that points rightwards from the $O$ marking and another line segment that connect to the $X$ marking on the same row through the identified edged on the toroidal diagram. By the above construction, if $\mathbb{G}$ is a grid diagram for a knot $K$, then braid closure of $B$ is $K$. This construction produces a rectilinear braid diagram and after perturbing it accordingly, we obtain a braid diagram. An example is given in Figure \ref{braid1}.

\begin{figure}[h!]
\centering
    \includegraphics[width=14cm]{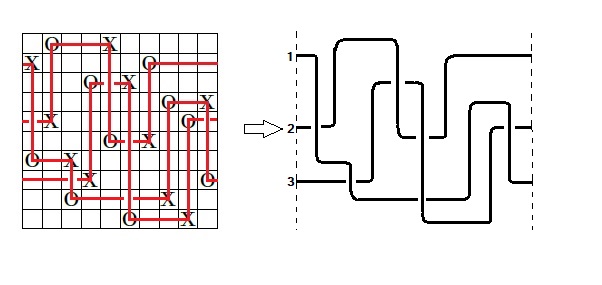}
    \caption{A braid in $\mathbf{B_3}$ with braid word $\sigma_1 \sigma_2 \sigma_1^2 \sigma_2^2$. This is also an example of a positive braid.}
    \label{braid1}
\end{figure} 

Any braid in $\mathbf{B_n}$ can be written as a product of the standard generators $\sigma_i^{\epsilon_i}$ for $i=1, \dots , n-1$ and $\epsilon_i \in \{ \pm 1\}$. In our convention, we enumerate the strands from top to bottom and a generator $\sigma_i$ refers to $i^{th}$ strand going over $(i+1)^{st}$ strand on the diagram. Notice also with this convention, $\epsilon_i$ agrees with the sign of the crossing. 

Now, we turn our focus to a particular family of braids, namely positive braids. A braid $B \in \mathbf{B_n}$ is called a positive braid if in its braid word all $\epsilon_i$ are positive. This is also equivalent to having all positive crossings. In \cite{braidfibered}, Stallings proves a more general result for homogeneous braids, which directly implies that positive braids are fibered. Combining this result with Hedden's result about strongly quasipositive knots (Proposition 2.1) in \cite{positivebraid}, we can conclude that the canonical grid state denoted as $\mathbf{x^+}$ is the element that computes $\tau$ and therefore, this is the element we need to keep track of when we compute $\varepsilon$ for positive braids.

\begin{proof}[Proof of Theorem \ref{braid}] 

Our proof relies on the observation that if a grid diagram $\mathbb{G}_+$ represents a positive braid, then we are limited to very particular local pictures and a relative $X$ and $O$-marking patterns. First of all, recall that with this convention, sign of the markings agree with $\epsilon_i$'s that appear in the braid word. Hence, we only have positive crossings on the grid diagram. Since all the horizontal strands are oriented from left to right and from $O$ to $X$, we have the local pictures as shown in Figure \ref{positivebraid}.

\begin{figure}[h!]
\centering
    \includegraphics[width=13.5cm]{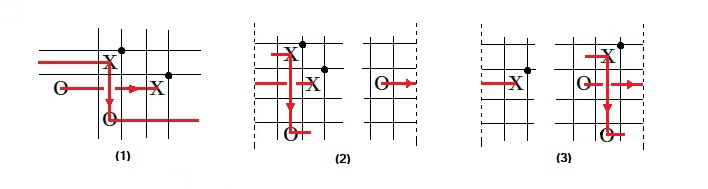}
    \caption{Different schematic local pictures of possible positive crossings. Black dots are intersection points from $\mathbf{x^+}$}
    \label{positivebraid}
\end{figure} 

Notice that the local pictures shown in Figure \ref{positivebraid} represent all possible positive crossing either they occur in the middle of the grid diagram (1), or they occur on horizontal strands that are joined outside the grid diagram (2) and (3). We mark the intersection points of $\mathbf{x^+}$ with black dots. We reflect the grid $\mathbb{G}_+$ with respect to a horizontal axis and change the markings. Similar to previous computations, the distinguished intersection points are now in the lower-right corners of $O$-markings. After this, we perform one local modification that allows us to compute epsilon. At the $O$-marking that is located at the bottom of the vertical strand, we apply an $O:NW$ type stabilization. As stated in Proposition 5.2.1 in \cite{ozsvath2015grid}, this does not change the homology, so we keep track of this distinguished element in the same homology after the stabilization. Moreover, stabilization creates a new pair of $\alpha$ and $\beta$ curve and their new intersection point is automatically added to all grid states, including $\mathbf{o_+}$. This is a one-to-one correspondence between grid states, hence we still use the same notation for this new grid state. There are exactly two rectangles from the grid state $\mathbf{o_+}$ to $\mathbf{y}$, one is the shaded rectangle in Figure \ref{positivebraidflip} and the other rectangle wraps around the torus in the complement. One can see that the shaded rectangle has non-trivial $V_i$ contribution for some $i$ whereas the other rectangle contain an $X$-marking. This implies that $\mathbf{y} \in \partial_{\Xmark}^{horz}(\mathbf{o_+}) $ and as a result, $[\mathbf{o_+}] \notin Ker(\partial_{\Xmark}^{horz})$. The same strategy also proves the existence of such a rectangle with non-trivial $V_i$ contribution for other local positive crossing pictures. Hence, we conclude $\varepsilon(\mathbb{G}_+)=1$.    

\begin{figure}[h!]
\centering
    \includegraphics[width=14cm]{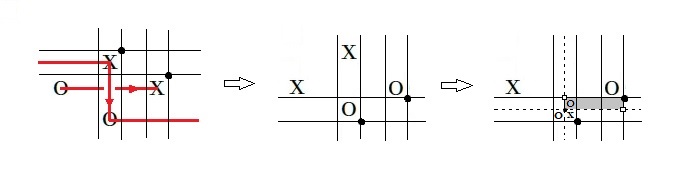}
    \caption{$O:NW$ stabilization performed on the grid in case (1). The grid state $\mathbf{y}$ is marked with small squares and the gray shaded area is the domain with non-zero contribution}
    \label{positivebraidflip}
\end{figure}

\end{proof}

\subsection*{Acknowledgements} The authors would like to thank \c Ca\u gatay Kutluhan, Andr\' as Stipsicz, John Etnyre, William Menasco, Lenhard Ng and Siddhi Krishna for their comments and suggestions on the first draft of the paper. We are especially thankful to Jen Hom for suggesting to look into positive braids and for her helpful comments. The first author would also like to acknowledge the support of the Dissertation Fellowship awarded by the Department of Mathematics at University at Buffalo during the project. The authors would also like to thank anonymous reviewers for their very helpful comments and suggestions. This work is partially supported by a Simons Foundation grant No.~519352.

\bibliographystyle{plain}
\bibliography{bibliography.bib}

\end{document}